\documentclass[a4paper]{amsart}
\usepackage{amsmath,amstext,amssymb,amsfonts,amscd,amsthm}
\usepackage[numbers,sort&compress]{natbib}
\usepackage{mathdots,mathrsfs,enumerate}
\usepackage{extarrows}
\usepackage{mathrsfs}  
\usepackage{dsfont}  
\usepackage{graphicx,color}
\usepackage{tikz}
\usepackage{stfloats}
\usetikzlibrary{3d,calc,patterns,graphs,arrows}
\usepackage{tikz-cd,array,diagbox}
\usepackage{changepage}
\usepackage{geometry}
\usepackage[CJKbookmarks=true,unicode,colorlinks,linkcolor=blue,anchorcolor=blue,citecolor=blue]{hyperref}
\usepackage{subfigure}
\numberwithin{equation}{section}

\usepackage{xcolor}
\definecolor{BeanPasteGreen}{RGB}{200,232,200}
\definecolor{WaterBlue}{RGB}{185,220,237}

\newtheorem{lemma}{Lemma}[section]
\newtheorem{theorem}{Theorem}[section]

\newtheorem{remark}{Remark}[section]

\newenvironment{proof*}
\numberwithin{equation}{section}

\newcommand{\op}[1]{\operatorname{#1}} 

\newcommand{\p}{\partial}

\renewcommand{\t}{\triangle}

\newcommand{\abs}[1]{\left\vert#1\right\vert}
\newcommand{\norm}[1]{\left\Vert#1\right\Vert}
\newcommand{\set}[1]{\left\{#1\right\}}

\newcommand{\frb}[1]{\left(#1\right)}

\newcommand{\ol}{\overline}
\newcommand{\ul}{\underline}

\newcommand{\ra}{\rightarrow}

\newcommand{\Z}{\mathbb Z}

\newcommand{\R}{\mathbb R}

\renewcommand{\a}{\alpha}
\renewcommand{\b}{\beta}
\newcommand{\g}{\gamma}
\providecommand{\G}{}
\renewcommand{\G}{\Gamma}
\renewcommand{\d}{\delta}
\newcommand{\D}{\Delta}

\newcommand{\ve}{\varepsilon}

\renewcommand{\l}{\lambda}

\newcommand{\s}{\sigma}

\newcommand{\vp}{\varphi}

\newcommand{\Om}{\Omega}

\allowdisplaybreaks[3]

\begin{document}

\title[Interior Hessian estimates for Hessian quotient equations]
{Interior Hessian estimates for $k$-convex solutions
of Hessian quotient equations with $2k>n$}

\author{Ke Wu}
\address{School of Mathematics and Center for Nonlinear Studies,
Northwest University, Xi'an, 710127, People's Republic of China}
\email{wuke@med.nwu.edu.cn}

\date{September 25, 2026.}

\begin{abstract}
We establish interior Hessian estimates for $k$-convex solutions of
$\frac{\s_k(D^2u)}{\s_\ell(D^2u)}=1$,
where $1\leq\ell<k<n$, $k-\ell\in\{1,2\}$ and $2k>n$.
The proof combines a concavity inequality
with a pointwise doubling argument.
\end{abstract}

\maketitle
\tableofcontents



\section{Introduction}\label{sec.introduction}

In this paper, we study interior Hessian estimates for the Hessian quotient equations
\begin{equation}\label{eqn.sigma-quotient}
F(D^2u):=\frac{\s_k(D^2u)}{\s_\ell(D^2u)}=1
\end{equation}
in dimension $n\geq2$, where $1\leq\ell<k\leq n-1$ and $k-\ell\in\{1,2\}$.
Here $\s_j(D^2u):=\s_j(\l(D^2u))$ denotes the $j$-th
elementary symmetric function of the eigenvalue-vector $\l(D^2u)$ of $D^2u$.

Following the classical theory of Caffarelli--Nirenberg--Spruck \cite{CNS85},
the natural elliptic branch for the $k$-Hessian operator is the G\aa rding cone
\[
\G_k:=\set{\l\in\R^n:\s_j(\l)>0\ \hbox{for all}\ 1\leq j\leq k}.
\]
The Hessian quotient operator $F$ is elliptic on $\G_k$.
In this setting, a function $u\in C^2(\Om)$ is called $k$-convex
if $\l(D^2u(x))\in\G_k$ for any $x\in\Om$.

\begin{theorem}\label{thm.Hessian-estimate}
Let $1\leq\ell<k\leq n-1$, $k-\ell\in\{1,2\}$ and $2k>n$.
Suppose that $u\in C^4(B_1)$ is a $k$-convex solution of
\eqref{eqn.sigma-quotient} in $B_1\subset\R^n$.
Then
\[
|D^2u(0)|\leq C\frb{n,k,\|u\|_{C^1(B_1)}}.
\]

From the interior gradient estimate for Hessian quotient equations by Chen \cite{C15},
we further obtain for $k$-convex solutions in $B_2$ that
\begin{equation}\label{eqn.Hessian-L-infty}
|D^2u(0)|
\leq C\frb{n,k,\norm{u}_{L^\infty(B_2)}}.
\end{equation}
\end{theorem}

\begin{remark}
The condition $2k>n$ is essentially technical and is used only in
establishing the strengthened concavity inequality and the resulting Jacobi inequality.
\end{remark}

\medskip

The study of interior Hessian estimates for Hessian equations has a long history.

For the $k$-Hessian equation $\s_k(D^2u)=f$,
the classical interior estimate goes back to Heinz \cite{H59}
for the Monge--Amp\`ere equation $\det(D^2u)=1$ in dimension two;
alternative proofs were later given by Chen--Han--Ou \cite{CHO16} and Liu \cite{Liu21}.
However, Pogorelov \cite{P78} constructed singular examples in dimensions $n\geq3$,
and Urbas \cite{U90} obtained analogous examples
for the $k$-Hessian equation $\s_k(D^2u)=f$ with $k\geq3$.

These counterexamples do not cover the quadratic case $k=2$.
In dimension three,
Warren--Yuan \cite{WY09} established the interior estimate for $\s_2(D^2u)=1$
and Qiu \cite{Q24} treated positive variable right-hand sides.
In dimension four,
Shankar and Yuan \cite{SY25} proved the interior Hessian estimate
for $\s_2(D^2u)=1$,
and also obtained corresponding estimates in higher dimensions
under the dynamic semiconvexity condition
\[\l_{\min}(D^2u)\geq-c(n)\t u,\]
while Fan \cite{Fan26} extended their results to variable right-hand sides.
In arbitrary dimensions, interior Hessian estimates were obtained
under almost convexity \cite{MSY19},
semiconvexity \cite{SY20},
or the condition $\s_3(D^2u)\geq-A$ \cite{GQ19}.
Interior regularity was also established
for convex viscosity solutions \cite{M21,CJTZ26,ZZ26}.
Recently,
Li--Wu \cite{LW26b} obtained interior Hessian estimates
for $\s_2(D^2u)=1$ in arbitrary dimensions
without additional convexity assumptions
by combining the Pogorelov-type estimate
of Chou--Wang \cite{CW01}
with a quantitative separation-propagation method.
Chen--Zhou--Zhu \cite{CZZ26} use this method
to study interior regularity for $\s_2(D^2u)=f(x)$
with positive $C^\alpha$ right-hand sides.

Interior Hessian estimates have also been studied for Hessian quotient equations
\[
\frac{\s_k(D^2u)}{\s_\ell(D^2u)}=f(x,u)
\quad\text{for}\ 1\leq\ell<k\leq n.
\]
For $\s_3/\s_1$ in dimensions three and four,
interior Hessian estimates follow from the special Lagrangian structure
\cite{CWY09,WY14,Z24};
Lu \cite{Lu23} subsequently gave a proof
based on the Jacobi inequality in dimension three.
Using a concavity inequality from \cite{GS26},
Lu \cite{Lu25} established interior estimates for $\s_n/\s_{n-1}$ and $\s_n/\s_{n-2}$.
He also constructed convex counterexamples
when $k-\ell\geq3$.
Jiao--Sui \cite{JS26} treated $\s_2/\s_1$
for $2$-convex solutions in dimension three
and semiconvex solutions in higher dimensions,
while Mei--Yan \cite{MY26} obtained estimates
for semiconvex solutions of $\s_3/\s_\ell=1$ for $\ell=1,2$ in arbitrary dimensions.
Li--Wu \cite{LW26b} obtained estimates for $2$-convex solutions
of $\s_2/\s_1=1$ in arbitrary dimensions
by subtracting a quadratic polynomial
and reducing the equation to $\s_2(D^2u)=1$.
For general Hessian quotients with $k-\ell\in\{1,2\}$,
Lu--Tsai \cite{LT26} derived interior estimates
for convex solutions under a structural concavity assumption.
This assumption was removed for $\ell=k-1$ by Tsai \cite{T26}
through a change of basis for symmetric polynomials,
and for both cases by Li--Wu \cite{LW26a}
through a contradiction argument and a one-dimensional lifting.
Around the same time,
Dong--Zhang \cite{DZ26} established corresponding estimates
for admissible semiconvex solutions
with positive $C^2$ right-hand sides $f(x,u)$.

\medskip

Two approaches to interior estimates for Hessian equations,
both based on Jacobi inequalities,
are closely related to our proof.

The first is the integral method of Warren--Yuan \cite{WY09},
refined by Qiu \cite{Q24} to avoid the Sobolev inequality
and further developed by Shankar--Yuan \cite{SY20}
for semiconvex solutions of the $\s_2$ equation.
In the quotient setting \cite{Lu25,LT26,DZ26},
concavity inequalities first control
the third-order terms in the largest-eigenvalue calculation,
yielding a Jacobi inequality for $\log\l_{\max}(D^2u)$.
A Lewy--Legendre transform then gives,
after normalization, a uniformly elliptic inequality.
The resulting mean-value estimate
reduces the pointwise bound to a weighted integral,
which is controlled by repeated integration by parts.

The second approach converts the Jacobi inequality
into a doubling inequality by the maximum principle.
Following Qiu's doubling estimate for the $\s_2$ equation \cite{Q24},
Shankar--Yuan \cite{SY25} developed a framework
combining doubling, small-perturbation theory \cite{Sav07}, and a compactness argument.
Shankar \cite{S26} applied this framework to the special Lagrangian equation
by combining the radial-derivative test function of Guan--Qiu \cite{GQ19}
with the Korevaar-type exponential cutoffs \cite{K87}.
Fung \cite{F26a} adapted this construction to Hessian quotient equations
with gradient-dependent right-hand sides
under convexity and suitable structural assumptions.
This route avoids the Legendre transform
and integration by parts.

\medskip

The proof of Theorem \ref{thm.Hessian-estimate}
combines ideas from the two approaches.
We establish the required concavity inequality,
derive a Jacobi inequality,
and then obtain the Hessian bound
through doubling and compactness arguments.
For $k$-convex solutions with $2k>n$,
carrying out this strategy presents two main difficulties.

The first difficulty is to establish the concavity inequality
(Lemma \ref{lem.concavity-quotient}).
Dong--Zhang \cite[Remark 2.1]{DZ26}
observed that their concavity inequality extends to
a dynamic semiconvexity regime of the form $\l_{\min}\geq-c\l_{\max}$.
It therefore remains to treat the case $\l_{\min}<-c\l_{\max}$.
We handle this case by a contradiction argument,
normalizing the eigenvalues by $\l_{\max}$,
and showing that the resulting boundary limit $\ol\mu$ satisfies $\p_1\s_k(\ol\mu)>0$.
This nondegeneracy, together with the tangency condition,
allows us to transfer Yan's concavity inequality \cite{Yan26}
to the quotient while retaining a strict margin.
The condition $2k>n$ makes this margin positive in the Jacobi calculation.
With this concavity inequality established,
the standard largest-eigenvalue calculation for solutions of $F(D^2u)=1$ gives
\begin{gather}
F^{ij}a_{ij}\geq2a^{-1}F^{ij}a_i a_j
\quad\text{for}\
a:=e^{\d(n,k) b},\
b:=\log\l_{\max}(D^2u),
\label{eqn.Jacobi-intro}
\end{gather}
in the viscosity sense when $\l_{\max}(D^2u)$ is sufficiently large,
where $F^{ij}:=\p F(D^2u)/\p u_{ij}$.

The second difficulty is to establish the doubling inequality:
\begin{equation}\label{eqn.doubling-intro}
\sup_{B_{2r}(y)}\l_{\max}(D^2u)
\leq C\frb{n,k,r,\|u\|_{C^1(B_1)}}\sup_{B_r(y)}\l_{\max}(D^2u)
\end{equation}
for any $y\in B_{1/2}$ and sufficiently small $r>0$.
To this end, we adapt the barrier construction
used in the separation-propagation argument of \cite[Section 3.2]{LW26b}.
We consider the function
\[
\vp_y(x)
:=(x-y)\cdot Du(x)-u(x)+u(y)+\frac{\a}{2}|x-y|^2-\b|Du(x)|^2,
\]
with the associated exponential cutoff
\[
\psi_y:=e^{(c_y-\vp_y)/\g}-1.
\]
The parameters are chosen so that
the component $\Om_y$ of $\{\vp_y<c_y\}$ containing $B_{2r}(y)$ satisfies
\[
B_{2r}(y)\subset\Om_y\Subset B_{4r}(y),
\quad
\psi_y\geq c_0>0\ \text{in}\ \ol{B_{2r}(y)},
\quad
\psi_y\leq C_0\ \text{in}\ \Om_y.
\]
To obtain \eqref{eqn.doubling-intro},
it suffices to show that a maximum point $x_0$ of $a\psi_y$ over $\ol{\Om_y}$
lies in $\ol{B_r(y)}$
whenever $\l_{\max}(D^2u(x_0))\geq K$.
Otherwise, the maximum principle and \eqref{eqn.Jacobi-intro} give
\[
F^{ij}(\psi_y)_{ij}(x_0)\leq0.
\]
On the other hand,
Lemma \ref{lem.Fj-lambdaj} gives $F_j\geq c\sum_iF_i$ whenever $\l_j<0$.
Combining this bound with the positive quadratic term in the eigenvalues generated by $-\b|Du|^2$,
we control the possible cancellation in $(\vp_y)_i$ caused by negative eigenvalues
and obtain $F^{ij}(\psi_y)_{ij}(x_0)>0$, a contradiction.
This gives \eqref{eqn.doubling-intro}.

Finally, in a compactness argument, 
we combine the Alexandrov-type result (Lemma \ref{lemma.alexandrov-quotient})
with Savin's small perturbation theorem \cite{Sav07}
to obtain a uniform Hessian bound 
for the approximating solutions on a fixed small ball.
Iterating the doubling inequality propagates this bound
to the point where blow-up is assumed, 
yielding a contradiction.

\medskip

\textit{Note on subsequent work.}
After the present work had been completed independently,
Qiu--Yan \cite{QY26} posted a preprint establishing
interior Hessian estimates for $k$-convex solutions of $\s_k(D^2u)/\s_\ell(D^2u)=1$
in arbitrary dimensions
for $2\leq k\leq n$, $0\leq\ell<k$, and $k-\ell\in\{1,2\}$,
without the restriction $2k>n$.

\section{Preliminaries}\label{sec.preliminaries}

In this section,
we fix notation,
recall basic algebraic properties of $\s_k$, $q_k:=\frac{\s_k}{\s_{k-1}}$ and $F$,
and introduce two auxiliary lemmas
that will be used in the subsequent proofs.
Throughout the paper, $C$ denotes a positive constant
that may change from line to line.

Let $\l=(\l_1,\l_2,\dots,\l_n)\in\R^n$.
For $i\in\Z_{[1,n]}$,
denote by $\l|i$ the vector obtained from $\l$ by deleting the $i$-th component.
By convention, $\s_i=0$ if $i<0$ or $i>n$, and $\s_0=1$.

For the Hessian quotient $F=\frac{\s_k}{\s_\ell}$
with $1\leq\ell<k\leq n$, we write
\[
F_i:=\frac{\p F}{\p\l_i}
\quad\text{and}\quad
F_{ij}:=\frac{\p^2F}{\p\l_i\p\l_j}.
\]
The normalized operator $F^{1/(k-\ell)}$
is elliptic, concave, and homogeneous of degree one
in $\G_k$; see \cite{CNS85}.
In particular,
\[
F_i(\l)
=\frac{\s_{k-1}(\l|i)}{\s_\ell(\l)}
-\frac{\s_k(\l)\s_{\ell-1}(\l|i)}{\s_\ell(\l)^2}
>0
\quad\text{for any}\ \l\in\G_k.
\]
Since $F$ is homogeneous of degree $k-\ell$,
Euler's identity and its derivative give
\[
\sum_{i=1}^n\l_iF_i=(k-\ell)F,
\quad
\sum_{i=1}^n\l_iF_{ij}=(k-\ell-1)F_j
\ \text{for any}\ j\in\Z_{[1,n]}.
\]
We also have
\begin{equation}\label{eqn.Fi-sum}
\sum_{i=1}^nF_i
=(n-k+1)\frac{\s_{k-1}(\l)}{\s_\ell(\l)}
-(n-\ell+1)
\frac{\s_k(\l)\s_{\ell-1}(\l)}{\s_\ell(\l)^2}.
\end{equation}

\medskip

Next, we introduce a lower bound for the smallest component of a vector in $\G_k$.

\begin{lemma}\label{eqn.lambda-i-lower}
Let $\l\in\G_k$ with $\l_1\geq\l_2\geq\cdots\geq\l_n$.
Then
\[
\l_i>-\frac{n-k}{k}\l_1\quad \text{for all}\  i\in\Z_{[1,n]}.
\]
\end{lemma}

\begin{proof}
See \cite[Lemma 11]{RW23}.
\end{proof}

\medskip

Finally, we establish the lower bounds
for the linearized coefficients of $F$,
which will be used to prove the doubling inequality (Lemma \ref{lem.doubling}).

\begin{lemma}\label{lem.Fj-lambdaj}
Let $1\leq\ell<k<n$, $k-\ell\in\set{1,2}$,
and let $\l\in\G_k$ satisfy $F(\l)=1$.
Then
\begin{equation}\label{eqn.Fj-lambdaj>}
F_j(1+\l_j^2)\geq c_1(n,k)\sum_{i=1}^n F_i\quad
\text{for any}\ j\in\Z_{[1,n]}.
\end{equation}
Moreover, if $\l_j<0$, then
\begin{equation}\label{eqn.Fj>}
F_j\geq c_2(n,k)\sum_{i=1}^nF_i,
\end{equation}
where $c_1=c_1(n,k)$ and $c_2=c_2(n,k)$ are positive constants.
\end{lemma}

\begin{proof}
We divide the proof into two cases.

\textit{Case 1. $F=\frac{\s_k}{\s_{k-1}}$}.
We first claim that
\begin{equation}\label{eqn.qk-i-bound}
\sum_{i=1}^nF_i\leq n-k+1,
\quad
F_j
\geq c\frb{\frac{\s_{k-1}(\l|j)}{\s_{k-1}}}^2
\ \text{for any}\ j\in\Z_{[1,n]},
\end{equation}
where $c=c(n,k)\in(0,1)$.
Indeed, using $\s_k=\s_{k-1}$ and \eqref{eqn.Fi-sum}, we have
\[
\sum_{i=1}^n(q_k)_i
=(n-k+1)-(n-k+2)\frac{\s_k\s_{k-2}}{\s_{k-1}^2}
\leq n-k+1.
\]
Fix any $j\in\Z_{[1,n]}$.
Using the identities
\[
\s_{k-1}=\s_{k-1}(\l|j)+\l_j\s_{k-2}(\l|j)
=\s_k
=\s_k(\l|j)+\l_j \s_{k-1}(\l|j),
\]
we have
\[
\s_{k-1}^2(q_k)_j=(\s_{k-1}(\l|j))^2-\s_{k-2}(\l|j)\s_k(\l|j).
\]
If $\s_k(\l|j)\leq0$, then $(q_k)_j\geq\frb{\frac{\s_{k-1}(\l|j)}{\s_{k-1}}}^2$.
If $\s_k(\l|j)>0$, the Newton--Maclaurin inequality in $n-1$ variables gives
\[
(q_k)_j
\geq\frb{1-\frac{(k-1)(n-k)}{k(n-k+1)}}\frb{\frac{\s_{k-1}(\l|j)}{\s_{k-1}}}^2.
\]
Hence \eqref{eqn.qk-i-bound} holds.

Next, if $\frac{\s_{k-1}(\l|j)}{\s_{k-1}}\geq\frac{1}{2}$,
then $(q_k)_j\geq\frac{c}{4}$.
If $\frac{\s_{k-1}(\l|j)}{\s_{k-1}}<\frac{1}{2}$,
then
\[
\s_{k-1}(\l|j)-\s_{k-2}(\l|j)=\s_{k-1}(q_k)_j>0,
\]
and hence
\[
\l_j^2(q_k)_j
\geq c\l_j^2\frb{\frac{\s_{k-2}(\l|j)}{\s_{k-1}}}^2
=c\frb{1-\frac{\s_{k-1}(\l|j)}{\s_{k-1}}}^2
>\frac{c}{4}.
\]
Combining these estimates with \eqref{eqn.qk-i-bound}, we obtain
\[
F_j(1+\l_j^2)
\geq\frac{c}{4}
\geq \frac{c}{4(n-k+1)}\sum_{i=1}^n F_i.
\]
Moreover, if $\l_j<0$, then $\l\in\G_k$ implies $\s_{k-2}(\l|j)>0$,
and hence $\s_{k-1}(\l|j)=\s_{k-1}-\l_j\s_{k-2}(\l|j)>\s_{k-1}$.
Thus, \eqref{eqn.qk-i-bound} gives
\[
F_j\geq c\geq \frac{c}{n-k+1}\sum_{i=1}^nF_i.
\]

\medskip

\textit{Case 2. $F=\frac{\s_k}{\s_{k-2}}$}.
Since $\s_k=\s_{k-2}$ and $\s_{k-3}(\l|i)\geq0$, we have
\begin{equation}\label{eqn.Fj-trace-k-2}
\sum_{i=1}^nF_i
=\sum_{i=1}^n\frac{\s_{k-2}\s_{k-1}(\l|i)-\s_k\s_{k-3}(\l|i)}{\s_{k-2}^2}
\leq\frac{\sum_{i=1}^n\s_{k-1}(\l|i)}{\s_{k-2}}
=(n-k+1)\frac{\s_{k-1}}{\s_{k-2}}.
\end{equation}
Applying the lower bound in \eqref{eqn.qk-i-bound} to $q_k$ and $q_{k-1}$
and decreasing $c$ if necessary, we obtain
\[
(q_k)_j\geq c\frb{\frac{\s_{k-1}(\l|j)}{\s_{k-1}}}^2,
\quad
(q_{k-1})_j\geq c\frb{\frac{\s_{k-2}(\l|j)}{\s_{k-2}}}^2.
\]
Using $F=q_kq_{k-1}$ and $\s_k=\s_{k-2}$, we obtain
\begin{equation}\label{eqn.Fj-square-k-2}
F_j=q_{k-1}(q_k)_j+q_k(q_{k-1})_j
\geq c\frac{\s_{k-1}(\l|j)^2+\s_{k-2}(\l|j)^2}
{\s_{k-1}\s_{k-2}}.
\end{equation}
On the other hand, the Cauchy--Schwarz inequality gives
\[
\s_{k-1}^2
=\frb{\s_{k-1}(\l|j)+\l_j\s_{k-2}(\l|j)}^2
\leq(1+\l_j^2)
\frb{\s_{k-1}(\l|j)^2+\s_{k-2}(\l|j)^2}.
\]
Combining this with \eqref{eqn.Fj-square-k-2} and
\eqref{eqn.Fj-trace-k-2}, we obtain
\[
F_j(1+\l_j^2)\geq c\frac{\s_{k-1}}{\s_{k-2}}
\geq\frac{c}{n-k+1}\sum_{i=1}^nF_i.
\]
Moreover, if $\l_j<0$, then $\s_{k-2}(\l|j)>0$ and hence $\s_{k-1}(\l|j)>\s_{k-1}$.
Thus, \eqref{eqn.Fj-square-k-2} and \eqref{eqn.Fj-trace-k-2} yield
\[
F_j\geq c\frac{\s_{k-1}(\l|j)^2}{\s_{k-1}\s_{k-2}}
\geq c\frac{\s_{k-1}}{\s_{k-2}}
\geq\frac{c}{n-k+1}\sum_{i=1}^nF_i.
\]
This completes the proof.
\end{proof}

\section{The concavity inequalities}
\label{sec.concavity}

In this section, we establish the concavity inequality for $F$
(Lemma \ref{lem.concavity-quotient}).
We first introduce two auxiliary concavity inequalities.

\begin{lemma}\label{lem.dynamic-concavity}
Let $1\leq\ell<k<n$, $k-\ell\in\set{1,2}$,
and let $\ve_0=\ve_0(n,k)\in(0,1)$ be sufficiently small.
Suppose that $\l\in\G_k$ satisfies
\[
\l_1>\l_2\geq\cdots\geq\l_n,
\quad
F(\l)=1.
\]
Then there exist $K_0=K_0(n,k)>1$ and $\d_0=\d_0(k)>0$ such that,
if
\[
\l_n\geq-\ve_0\l_1
\quad\text{and}\quad
\l_1\geq K_0,
\]
then
\[
-\sum_{i,j=1}^nF_{ij}\xi_i\xi_j
+\frac2{(1+\ve_0)\l_1}\sum_{i=2}^n F_i\xi_i^2
\geq(1+\d_0)\frac{F_1\xi_1^2}{\l_1}
\]
for any $\xi\in\R^n$ with $DF(\l)\cdot\xi=0$.
\end{lemma}

\begin{proof}
See \cite[Remark 2.1]{DZ26} and \cite[Lemma 3.1]{DXZ26}.
\end{proof}

The above inequality requires the dynamic semiconvexity assumption.
To handle the case $\l_n<-\ve_0\l_1$
in the proof of Lemma \ref{lem.concavity-quotient},
we also need the following concavity inequality for $\s_k$.

\begin{lemma}\label{lem.concavity-sigma-k}
Let $2\leq k\leq n-1$
and let $\l\in\G_k$ with $\l_1>\l_2\geq\cdots\geq\l_n$.
Assume that
\[
0<\g<\min\set{\frac{2k}{n},1+\frac{2k-n}{2k^2+n}}.
\]
Then there exists a small constant $\ve=\ve(n,k,\g)>0$ such that
whenever
\[
\l_1\geq\s_k(\l)^{1/k}/\ve,
\]
we have
\begin{equation}\label{eqn.sigma-concavity}
-\sum_{i,j}(\s_k)_{ij}\xi_i\xi_j
+\frac2{\s_k}\frb{\sum_i(\s_k)_i\xi_i}^2
+2\sum_{i=2}^n\frac{(\s_k)_i\xi_i^2}{\l_1-\l_i}
\geq\g\frac{(\s_k)_1\xi_1^2}{\l_1}
\end{equation}
for every $\xi\in\R^n$.
\end{lemma}

\begin{proof}
See \cite[Theorem 1.1]{Yan26}.
\end{proof}

\medskip

The following algebraic lemma will be used to show that
the limiting vector in the contradiction argument lies in $\G_{k-1}$.

\begin{lemma}\label{lem.singular-projection}
Let $2\leq k\leq n$
and let $\l=(1,\l_2,\ldots,\l_n)\in\ol{\G_k}$ satisfy $\s_k(\l)=0$.
If $\p_1\s_k(\l)=0$, then $\l\in\ol{\G_n}$
and $\l$ has at most $k-1$ nonzero coordinates.
\end{lemma}

\begin{proof}
Set $\ul\l:=(\l|1)\in\R^{n-1}$.
Since $\l\in\ol{\G_k}$, we have $\ul\l\in\ol{\G_{k-1}}$.
The assumptions give
\[
0=\p_1\s_k(\l)
=\s_{k-1}(\ul\l),
\quad
0=\s_k(\l)
=\s_{k-1}(\ul\l)+\s_k(\ul\l)
=\s_k(\ul\l).
\]
Thus
\begin{equation}\label{eqn.weighted-Newton}
\sum_{i=2}^n \ul\l_i^2\s_{k-2}(\ul\l| i)
=\s_1(\ul\l)\s_{k-1}(\ul\l)-k\s_k(\ul\l)=0.
\end{equation}
Each summand is nonnegative, so
$\ul\l_i\ne0$ implies $\s_{k-2}(\ul\l| i)=0$ for all $i\in\Z_{[2,n]}$.
If $k=2$, then $\ul\l=0$, and the conclusion follows.

Assume now that $k\geq3$.
Let $\mu=(\mu_1,\ldots,\mu_r)$ be the vector consisting of
all nonzero coordinates of $\ul\l$.
Suppose to the contrary that $r\geq k-1$.
Then $\s_{k-2}(\mu| i)=0$ for all $i\in\Z_{[1,r]}$,
and hence $\s_{k-2}(\mu)=0$.
Furthermore, $\s_{k-3}(\mu|i)=\frac{\s_{k-2}(\mu)-\s_{k-2}(\mu|i)}{\mu_i}=0$ for all $i\in\Z_{[1,r]}$,
and thus $\s_{k-3}(\mu)=0$.
Iteration leads to $\s_0(\mu|i)=0$, a contradiction.
Thus $r\leq k-2$,
and $\ul\l$ has at most $k-2$ nonzero coordinates.

Finally, since $\ul\l\in\ol{\G_{k-1}}$ and $r\leq k-2$, we have
\[
\prod_{i=1}^r(t+\mu_i)
=\sum_{j=0}^r\s_j(\mu)t^{r-j}>0
\quad\text{for any}\ t>0.
\]
This polynomial has no positive root.
Since each $\mu_i$ is nonzero,
we have $\mu_i>0$ for all $i\in\Z_{[1,r]}$.
Together with $\l_1=1$, this shows that $\l\in\ol{\G_n}$.
\end{proof}

\medskip

We now use the above three lemmas to
prove the following concavity inequality for $F$ under $k$-convexity alone.

\begin{lemma}[Key lemma]\label{lem.concavity-quotient}
Let $1\leq\ell<k\leq n-1$, $k-\ell\in\set{1,2}$ and $2k>n$.
Suppose that $\d_0=\d_0(k)>0$ is the constant given in Lemma \ref{lem.dynamic-concavity},
and $\d=\d(n,k)$ satisfies
\begin{equation}\label{eqn.delta-upper}
0<2\d<\min\set{2\d_0,\frac{2k}{n}-1,\frac{2k-n}{2k^2+n}}.
\end{equation}
Then there exists $K=K(n,k)>1$ such that,
for any $\l=(\l_i)\in\G_k$ satisfying
\[
\l_1>\l_2\geq\cdots\geq\l_n,
\quad
F(\l)=1,
\quad
\l_1\geq K,
\]
we have
\begin{equation}\label{eqn.quotient-concavity}
-\sum_{i,j=1}^nF_{ij}\xi_i\xi_j
+2\sum_{i=2}^n\frac{F_i\xi_i^2}{\l_1-\l_i}
\geq(1+\d)\frac{F_1\xi_1^2}{\l_1}
\end{equation}
for every $\xi\in\R^n$ satisfying $DF(\l)\cdot\xi=0$.
\end{lemma}

\begin{proof}
Suppose to the contrary that
there exist sequences $\l^{\nu}\in\G_k$ and $\xi^{\nu}\in\R^n$ such that
\begin{gather}
\l^\nu_1>\l^\nu_2\geq\cdots\geq\l^\nu_n,
\quad
F(\l^\nu)=1,
\quad
\l_1^\nu\geq\nu\to+\infty,
\quad
DF(\l^\nu)\cdot\xi^\nu=0,
\nonumber\\
-\p_{\xi^\nu}^2F(\l^\nu)
+2\sum_{i=2}^n\frac{F_i(\l^\nu)(\xi^\nu_i)^2}{\l_1^\nu-\l_i^\nu}
<(1+\d)\frac{F_1(\l^\nu)(\xi^\nu_1)^2}{\l_1^\nu}.
\label{eqn.concavity-contrary}
\end{gather}
By the concavity of $F^{1/(k-\ell)}$
and the homogeneity in $\xi^\nu$,
we may assume that $\xi_1^\nu=1$.
Set
\[
\ve^\nu:=\frac{1}{\l^\nu_1},
\quad
\mu^\nu:=\ve^\nu\l^\nu,
\quad
\eta^\nu:=(\ve^\nu)^{k-\ell}.
\]
By homogeneity, we have
\[
\mu^\nu_1=1,
\quad
F(\mu^\nu)=\eta^\nu\to0,
\quad
DF(\mu^\nu)\cdot\xi^\nu=0.
\]
By Lemma \ref{eqn.lambda-i-lower}, it follows that
\begin{equation}\label{eqn.mu-i-lower-bound}
\mu^\nu_i
>-\frac{n-k}{k}\mu^\nu_1
=-\frac{n-k}{k}
\quad\text{for all}\ i\in\Z_{[1,n]}.
\end{equation}
After passing to a subsequence,
we may assume that $\mu^\nu\to\ol\mu\in\ol{\G_k}$.
Since
\[
\s_k(\mu^\nu)=\eta^\nu\s_\ell(\mu^\nu)\to0,
\]
we have $\bar\mu_1=1$ and $\s_k(\bar\mu)=0$.
Next,
let $\ve_0$ and $K_0=K_0(n,k)$ be as in Lemma \ref{lem.dynamic-concavity} for $F$.
We consider two cases separately.

\medskip

\textit{Case 1. $\ol\mu_n\geq-\ve_0/2$.}\
For sufficiently large $\nu$,
we have $\l_n^{\nu}\geq-\ve_0\l_1^{\nu}$ and $\l_1^\nu\geq K_0$.
In particular, $\l_1^{\nu}-\l_i^{\nu}\leq(1+\ve_0)\l_1^{\nu}$ for all $i\in\Z_{[2,n]}$.
Thus, we can apply Lemma \ref{lem.dynamic-concavity} to get
\begin{align*}
-\p_{\xi^\nu}^2F(\l^\nu)
+2\sum_{i=2}^n\frac{F_i(\l^\nu)(\xi^\nu_i)^2}{\l_1^\nu-\l_i^\nu}
&\geq-\p_{\xi^\nu}^2F(\l^\nu)
+\frac2{(1+\ve_0)\l_1^\nu}\sum_{i=2}^nF_i(\l^\nu)(\xi^\nu_i)^2
\\
&\geq(1+\d_0)\frac{F_1(\l^\nu)(\xi^\nu_1)^2}{\l_1^\nu},
\end{align*}
which contradicts the inequality \eqref{eqn.concavity-contrary}.

\textit{Case 2. $\ol\mu_n<-\ve_0/2$.}
Lemma \ref{lem.singular-projection} implies that $\p_1\s_k(\ol\mu)>0$,
and hence $\s_{k-1}(\ol\mu)>0$, and $\ol\mu\in\G_{k-1}$.
By the inequality \eqref{eqn.concavity-contrary},
the concavity of $F^{1/(k-\ell)}$ and $F_i(\l^\nu)\geq F_1(\l^\nu)>0$,
for any $i\in\Z_{[2,n]}$, we have
\[
2F_1(\l^\nu)\frac{(\xi^\nu_i)^2}{\l_1^\nu-\l_i^\nu}
\leq2\sum_{i=2}^n\frac{F_i(\l^\nu)(\xi^\nu_i)^2}{\l_1^\nu-\l_i^\nu}
<(1+\d)\frac{F_1(\l^\nu)}{\l_1^\nu},
\]
and hence
\begin{equation}\label{eqn.xi-nu<}
\frac{(\xi^\nu_i)^2}{1-\mu_i^\nu}\leq\frac{1+\d}{2}.
\end{equation}
Combining this with \eqref{eqn.mu-i-lower-bound}, we obtain
\begin{equation}\label{eqn.xi-nu<C}
|\xi^\nu|^2-1
=\sum_{i=2}^n(\xi^\nu_i)^2
\leq\sum_{i=2}^n\frac{1+\d}{2}(1-\mu_i^\nu)
\leq\sum_{i=2}^n\frac{1+\d}{2}\frb{1+\frac{n-k}{k}}
\leq\frac{n(n-1)(1+\d)}{2k}.
\end{equation}
Differentiating $\s_k=F\s_\ell$
and using $DF(\mu^\nu)\cdot\xi^\nu=0$, we obtain
\begin{align*}
\s_k(\mu^\nu)&=F(\mu^\nu)\s_\ell(\mu^\nu)=\eta^\nu\s_\ell(\mu^\nu),
\\
\s_\ell(\mu^\nu)F_i(\mu^\nu)
&=\p_i\s_k(\mu^\nu)-\eta^\nu\p_i\s_\ell(\mu^\nu),\
\\
\s_\ell(\mu^\nu)\p_{\xi^\nu}^2F(\mu^\nu)
&=\p_{\xi^\nu}^2\s_k(\mu^\nu)-\eta^\nu\p_{\xi^\nu}^2\s_\ell(\mu^\nu),
\\
\p_{\xi^\nu}\s_k(\mu^\nu)
&=\eta^\nu\p_{\xi^\nu}\s_\ell(\mu^\nu).
\end{align*}
Since $\mu_1^\nu=1$ and $\s_k(\mu^\nu)\to0$,
Lemma \ref{lem.concavity-sigma-k} applies
with $\l=\mu^\nu$, $\xi=\xi^\nu$, and $\g=1+2\d$
for all sufficiently large $\nu$.
Using homogeneity and combining this with \eqref{eqn.xi-nu<}, \eqref{eqn.xi-nu<C}
and $|\mu^\nu_i|\leq1$ for all $i\in Z_{[1,n]}$,
we obtain
\begin{align*}
0&>\s_\ell(\mu^\nu)\frb{-\p_{\xi^\nu}^2F(\mu^\nu)
+2\sum_{i=2}^n\frac{F_i(\mu^\nu)(\xi^\nu_i)^2}{1-\mu_i^{\nu}}
-(1+\d)F_1(\mu^\nu)}
\\
&=-\p_{\xi^\nu}^2\s_k(\mu^\nu)+\eta^\nu\p_{\xi^\nu}^2\s_\ell(\mu^\nu)
+2\sum_{i=2}^n\frac{\frb{\p_i\s_k(\mu^\nu)-\eta^\nu\p_i\s_\ell(\mu^\nu)}(\xi^\nu_i)^2}{1-\mu_i^{\nu}}
\\
&\quad\,
-(1+\d)\frb{\p_1\s_k(\mu^\nu)-\eta^\nu\p_1\s_\ell(\mu^\nu)}
\\
&=-\p_{\xi^\nu}^2\s_k(\mu^\nu)
+2\sum_{i=2}^n\frac{\p_i\s_k(\mu^\nu)(\xi^\nu_i)^2}{1-\mu_i^{\nu}}
-(1+2\d)\p_1\s_k(\mu^\nu)+\d\p_1\s_k(\mu^\nu)
\\
&\quad\,
+\eta^\nu\frb{
\p_{\xi^\nu}^2\s_\ell(\mu^\nu)
-2\sum_{i=2}^n\frac{\p_i\s_\ell(\mu^\nu)(\xi^\nu_i)^2}{1-\mu_i^{\nu}}
+(1+\d)\p_1\s_\ell(\mu^\nu)
}
\\
&\geq-\frac{2}{\s_k(\mu^\nu)}(\p_{\xi^\nu}\s_k(\mu^\nu))^2
+\d\p_1\s_k(\mu^\nu)-C(n,k)\eta^\nu
\\
&=\d\p_1\s_k(\mu^\nu)
-2\eta^\nu\frac{(\p_{\xi^\nu}\s_\ell(\mu^\nu))^2}{\s_\ell(\mu^\nu)}-C(n,k)\eta^\nu
\\
&\geq\frac{\d}{2}\p_1\s_k(\ol\mu)
>0
\end{align*}
for all large $\nu$.
Here we used $\eta^\nu\to0$,
the boundedness of $\xi^\nu$,
$\s_\ell(\mu^\nu)\to\s_\ell(\ol\mu)>0$
and $\p_1\s_k(\mu^\nu)\to\p_1\s_k(\ol\mu)>0$.
This contradiction completes the proof.
\end{proof}

\section{The Jacobi inequality and the doubling inequality}\label{sec.Jacobi}

In this section,
we first combine Lemma \ref{lem.concavity-quotient}
with the standard largest-eigenvalue calculation
(see also \cite[Lemma 4.1]{Lu25})
to derive the Jacobi inequality.
We then use this inequality,
together with a modification of the pointwise doubling method
introduced in \cite[Proposition 4.2]{S26}
and subsequently employed in \cite[Section 4]{F26a},
to establish the doubling inequality.

Throughout this section,
we assume that $1\leq\ell<k\leq n-1$,  $k-\ell\in\set{1,2}$ and $2k>n$.
For the operator $F=\frac{\s_k}{\s_\ell}$,
we use the notation
\[
F^{ij}
:=\frac{\p F}{\p u_{ij}}(D^2u),
\quad
F^{pq,rs}
:=\frac{\p^2F}{\p u_{pq}\p u_{rs}}(D^2u).
\]
At a point where $D^2u$ is diagonal
with eigenvalues $\l_1,\l_2,\dots,\l_n$,
we have
\[
F^{ij}=F_i\d_{ij}\quad\text{and}\quad F^{ii,jj}=F_{ij}.
\]

\medskip

With this notation in place,
we first establish the following Jacobi inequality.

\begin{lemma}[Jacobi inequality]\label{lem.Jacobi}
Suppose that $\d=\d(n,k)\in(0,1)$ and $K=K(n,k)$
are the constants given in Lemma \ref{lem.concavity-quotient}.
Assume that $u\in C^4(B_1)$ is a $k$-convex solution of \eqref{eqn.sigma-quotient} in $B_1$.
Let $b(x):=\log\l_{\max}(D^2u(x))$ and $a(x):=\l_{\max}(D^2u(x))^\d$.
Then
\begin{equation}\label{eqn.jacobi}
F^{ij}b_{ij}\geq\d F^{ij}b_i b_j
\quad\text{and}\quad
F^{ij}a_{ij}
\geq2a^{-1}F^{ij}a_i a_j
\end{equation}
hold in the viscosity sense
at every point $x\in B_1$ satisfying $\l_{\max}(D^2u)\geq K$.
\end{lemma}

\begin{proof}
Fix $x_0\in B_1$ such that $\l_{\max}(D^2u(x_0))\geq K$.
After a rotation of coordinates, we may assume that
\[
D^2u(x_0)=\op{diag}(\l_1,\ldots,\l_n),
\quad
\l_1=\cdots=\l_m=\l_{\max}>
\l_{m+1}\geq\cdots\geq\l_n.
\]
All quantities below are evaluated at $x_0$.
By \cite[Lemma 5]{BCD17}, it follows that
\begin{gather}
\d_{pq}(\l_1)_i=u_{pqi}
\quad\text{for all}\
p,q\in\Z_{[1,m]},\ i\in\Z_{[1,n]}
\label{eqn.lambda-max-derivative}
\\
(\l_1)_{ii}
\geq u_{11ii}+2\sum_{p>m}\frac{u_{1pi}^2}{\l_1-\l_p}
\quad\text{for all}\ i\in\Z_{[1,n]}
\label{eqn.lambda-max-second}
\end{gather}
in the viscosity sense.
Moreover, we can compute
\begin{gather*}
u_{11i}=\l_1b_i\quad\text{for all}\ i\in\Z_{[m+1,n]},\
\\
b_{ii}
\geq\frac{u_{11ii}}{\l_1}
+2\sum_{p>m}\frac{u_{1pi}^2}{\l_1(\l_1-\l_p)}-\frac{u_{11i}^2}{\l_1^2}
\quad\
\text{for all}\ i\in\Z_{[1,n]}
\end{gather*}
in the viscosity sense.
Hence
\begin{equation}\label{eqn.Fij-bij}
F^{ij}b_{ij}
\geq\frac{1}{\l_1}\sum_{i=1}^nF_i u_{11ii}
+2\sum_{i=1}^n\sum_{p>m}\frac{F_i u_{1pi}^2}{\l_1(\l_1-\l_p)}
-\sum_{i=1}^n\frac{F_i u_{11i}^2}{\l_1^2}.
\end{equation}
Set $\xi_i:=u_{ii1}$ for all $i\in\Z_{[1,n]}$.
Differentiating the equation $F(D^2u)=1$, we have
\[
DF\cdot\xi=\sum_{i=1}^nF_iu_{ii1}=\p_{x_1}F=0.
\]
Differentiating the equation $F(D^2u)=1$ twice in the $x_1$-direction,
and using the concavity of $F^{1/(k-\ell)}$
and the standard second-derivative formula for symmetric functions
of the eigenvalues (see \cite[Lemma 2.2]{LT26}), we obtain
\begin{align}
\sum_{i=1}^nF_i u_{ii11}
&=-\sum_{p,q,r,s=1}^nF^{pq,rs}u_{pq1}u_{rs1}
=-\sum_{i,j=1}^nF_{ij}\xi_i\xi_j-\sum_{i\neq j}F^{ij,ji}u_{ij1}^2
\nonumber\\
&\geq-\sum_{i,j=1}^nF_{ij}\xi_i\xi_j+2\sum_{p>m}\frac{F_p-F_1}{\l_1-\l_p}u_{11p}^2.
\label{eqn.Fi-uii11}
\end{align}
On the other hand,
\begin{equation}\label{eqn.third-derivative-lower-jacobi}
2\sum_{i=1}^n\sum_{p>m}
\frac{F_i u_{1pi}^2}{\l_1(\l_1-\l_p)}
\geq2\sum_{p>m}\frac{F_1u_{11p}^2+F_pu_{1pp}^2}{\l_1(\l_1-\l_p)}
=2\sum_{p>m}\frac{F_1u_{11p}^2+F_p\xi_p^2}{\l_1(\l_1-\l_p)}.
\end{equation}
Substituting
\eqref{eqn.Fi-uii11} and \eqref{eqn.third-derivative-lower-jacobi}
into \eqref{eqn.Fij-bij}, we find
\begin{align*}
F^{ij}b_{ij}
&\geq\frac{1}{\l_1}
\frb{-\sum_{i,j=1}^nF_{ij}\xi_i\xi_j+2\sum_{p>m}\frac{F_p-F_1}{\l_1-\l_p}u_{11p}^2}
+2\sum_{p>m}\frac{F_1u_{11p}^2+F_p\xi_p^2}{\l_1(\l_1-\l_p)}
-\sum_{i=1}^n\frac{F_i u_{11i}^2}{\l_1^2}
\\
&=\frac{1}{\l_1}
\frb{-\sum_{i,j=1}^nF_{ij}\xi_i\xi_j+\sum_{p>m}\frac{2F_p\xi_p^2}{\l_1-\l_p}}
+2\sum_{p>m}\frac{F_p u_{11p}^2}{\l_1(\l_1-\l_p)}-\sum_{i=1}^n\frac{F_i u_{11i}^2}{\l_1^2}.
\end{align*}
Next, we consider two cases for $m$.

\textit{Case 1. $m>1$.}\
By \eqref{eqn.lambda-max-derivative}, it follows that
\begin{gather*}
\xi_1=u_{111}=u_{221}=u_{122}=\d_{12}(\l_1)_2=0,
\\
\xi_i=u_{ii1}=u_{1ii}=\d_{1i}(\l_1)_i=0\quad\text{for all}\ i\in\Z_{[2,m]},
\\
b_i=\frac{u_{11i}}{\l_1}=\frac{u_{1i1}}{\l_1}=\frac{\d_{1i}(\l_1)_1}{\l_1}=0
\quad\text{for all}\ i\in\Z_{[1,m]}.
\end{gather*}
Since $(\s_k/\s_{k-2})^{1/2}$ is concave, we have
\[
\p_\xi^2F(\l)
=2F^{1/2}\p_\xi^2(F^{1/2})
+\frac{(DF\cdot\xi)^2}{2F}
\leq0
\quad\text{for}\ F=\frac{\s_k}{\s_{k-2}}.
\]
Moreover, $\s_k/\s_{k-1}$ is also concave.
Thus $-\sum_{i,j=1}^nF_{ij}\xi_i\xi_j\geq0$ and therefore
\begin{align*}
F^{ij}b_{ij}
&\geq2\sum_{p>m}\frac{F_p u_{11p}^2}{\l_1(\l_1-\l_p)}-\sum_{i>m}\frac{F_i u_{11i}^2}{\l_1^2}
=\sum_{i>m}\frb{\frac{2\l_1}{\l_1-\l_i}-1}F_ib_i^2
\\
&\geq\sum_{i>m}\frb{\frac{2\l_1}{\l_1+(n-k)\l_1/k}-1}F_ib_i^2
=\frb{\frac{2k}{n}-1}\sum_{i,j=1}^nF^{ij}b_ib_j
\\
&\geq\d\sum_{i,j=1}^nF^{ij}b_ib_j,
\end{align*}
where the second inequality follows from $\l_i\geq-(n-k)\l_1/k$ for all $i\in\Z_{[2,n]}$.

\textit{Case 2. $m=1$.}\
By \eqref{eqn.lambda-max-derivative},
it follows that $\xi_1=u_{111}=(\l_1)_1=\l_1b_1$.
Since $\l_1\geq K$,
the concavity inequality (Lemma \ref{lem.concavity-quotient}) gives
\[
-\sum_{i,j=1}^nF_{ij}\xi_i\xi_j
+\sum_{i>m}\frac{2F_i\xi_i^2}{\l_1-\l_i}
\geq(1+\d)\frac{F_1\xi_1^2}{\l_1}.
\]
Consequently,
\begin{align*}
F^{ij}b_{ij}
&\geq(1+\d)F_1b_1^2
+2\sum_{i>1}\frac{\l_1F_i}{\l_1-\l_i}b_i^2
-\sum_{i=1}^nF_i b_i^2
\\
&=\d F_1b_1^2+\sum_{i>1}\frb{\frac{2\l_1}{\l_1-\l_i}-1}F_i b_i^2
\geq\d \sum_{i,j=1}^nF^{ij}b_ib_j.
\end{align*}

Finally, since $a=e^{\d b}$, we have
\[
F^{ij}a_{ij}
=\d a F^{ij}b_{ij}+\d^2aF^{ij}b_ib_j
\geq2\d^2aF^{ij}b_i b_j
=2a^{-1}F^{ij}a_i a_j.
\]
This completes the proof.
\end{proof}

\medskip

Combining the Jacobi inequality
with Lemma \ref{lem.Fj-lambdaj},
we obtain the following doubling inequality.

\begin{lemma}[Doubling inequality]\label{lem.doubling}
Suppose that $u\in C^4(B_1)$ is a $k$-convex solution of $F(D^2u)=1$ in $B_1$.
Then for any $y\in B_{1/2}$ and any $r\in(0,1/8)$,
we have
\begin{equation}\label{eqn.doubling}
\sup_{B_{2r}(y)}\l_{\max}(D^2u)
\leq C\frb{n,k,r,\|u\|_{C^1(B_1)}}\sup_{B_r(y)}\l_{\max}(D^2u).
\end{equation}
\end{lemma}

\begin{proof}
Suppose that $\d=\d(n,k)>0$ and $K=K(n,k)>1$ are the constants in Lemma \ref{lem.Jacobi}.
Fix $y\in B_{1/2}$ and $r\in(0,1/8)$, so that $B_{4r}(y)\Subset B_1$.
Consider the auxiliary function
\[
\vp_y(x):=(x-y)\cdot Du(x)-u(x)+u(y)+\frac{\a}{2}|x-y|^2-\b\abs{Du}^2,
\]
where $\a=\a(n,k,r,M)\geq1$ and $\b=\b(n,r,M)>0$ will be chosen later.
Let $M:=\|u\|_{C^1(B_1)}+1$.
The gradient bound gives
\begin{align*}
\abs{(x-y)\cdot Du(x)-u(x)+u(y)}\leq2M|x-y|\quad\ \text{for any}\ x\in B_1.
\end{align*}
Then
\begin{gather*}
\vp_y\geq-8Mr-\b M^2\quad\text{in}\ B_{4r}(y),\quad
a_y:=\sup_{\ol{B_{2r}(y)}}\vp_y
\leq4Mr+2\a r^2.
\end{gather*}
Choosing $\a$ such that $6r^2\a-M^2\b-12Mr-2>0$,
we have
\[
\inf_{\p B_{4r}(y)}\vp_y
\geq-8Mr+\frac{\a}{2}(4r)^2-\b M^2
>4Mr+2\a r^2+2
\geq a_y+2=:c_y.
\]
Then
\[
\vp_y<c_y
\quad\text{in}\ \ol{B_{2r}(y)},
\quad
\vp_y>c_y
\quad\text{on}\ \p B_{4r}(y).
\]
Thus the connected component $\Om_y$ of
$\{x\in B_{4r}(y):\vp_y(x)<c_y\}$
containing $B_{2r}(y)$ satisfies
\begin{align*}
\ol{B_{2r}(y)}\subset\Om_y\Subset B_{4r}(y).
\end{align*}

Next, define the associated exponential auxiliary function
\begin{align*}
\psi_y:=e^{(c_y-\vp_y)/\g}-1
\quad\text{in}\ \Om_y,
\end{align*}
where $\g=\g(n,k,r,M)>0$ will be chosen later.
Then $\psi_y>0$ in $\Om_y$ and $\psi_y=0$ on $\p\Om_y$.
Moreover,
\begin{gather}
\psi_y\geq e^{1/\g}-1=:c_0>0
\quad\text{in}\ \ol{B_{2r}(y)}.
\label{eqn.psi-lower}
\\
\psi_y\leq
\exp\frb{\frac{12Mr+2\a r^2+\b M^2+2}{\g}}-1
=:C_0
\quad\text{in}\ \Om_y.
\label{eqn.psi-upper}
\end{gather}
Let $a:=\l_{\max}(D^2u)^\d$,
and let $x_0\in\Om_y$ be a maximum point of $a\psi_y$ in $\ol{\Om_y}$.
We claim that
\[
x_0\in\ol{B_r(y)}
\quad\text{or}\quad
\l_{\max}(D^2u(x_0))<K.
\]
Suppose to the contrary that $|x_0-y|>r$ and $\l_{\max}(D^2u(x_0))\geq K$.
All subsequent calculations are at $x_0$.
Since $x_0$ is a maximum point,
we have
\[
0=(a\psi_y)_i=a_i\psi_y+a(\psi_y)_i,
\quad
0\geq(a\psi_y)_{ij}.
\]
By the Jacobi inequality (Lemma \ref{lem.Jacobi}), it follows that
\begin{align}
0\geq F^{ij}(a\psi_y)_{ij}
&=\psi_yF^{ij}a_{ij}+aF^{ij}(\psi_y)_{ij}+2F^{ij}a_i(\psi_y)_j
\nonumber\\
&\geq\psi_y\frac{2}{a}F^{ij}a_ia_j
+aF^{ij}(\psi_y)_{ij}-2F^{ij}a_i\frac{a_j\psi_y}{a}
\nonumber\\
&=aF^{ij}(\psi_y)_{ij}.
\label{eqn.Fij-psi<0}
\end{align}
On the other hand,
after a rotation of coordinates, we may assume that
\[
D^2u(x_0)
=\op{diag}(\l_1,\l_2,\ldots,\l_n),
\quad
\l_1\geq\l_2\geq\cdots\geq\l_n.
\]
Using $F=1$ and $\sum_{i,j=1}^nF^{ij}u_{ijl}=\p_{x_l}F=0$ for all $l\in\Z_{[1,n]}$,
we compute
\begin{align*}
(\vp_y)_i
&=\frb{x_i-y_i-2\b u_i}\l_i+\a(x_i-y_i),
\\
(\vp_y)_{ij}
&=u_{ij}+\sum_{l=1}^n(x_l-y_l)u_{ijl}+\a\d_{ij}
-2\b\sum_{l=1}^n u_{lj}u_{li}
-2\b\sum_{l=1}^nu_lu_{ijl},
\\
F^{ij}(\vp_y)_{ij}
&=k-\ell+\a\sum_{i=1}^nF_i(\l)-2\b\sum_{i=1}^nF_i\l_i^2,
\\
F^{ij}(\psi_y)_{ij}
&=\frac{\psi_y+1}{\g}
\frb{\frac{1}{\g}\sum_{i=1}^nF_i(\vp_y)_i^2
+2\b\sum_{i=1}^nF_i\l_i^2-(k-\ell)-\a\sum_{i=1}^nF_i}.
\end{align*}
Together with \eqref{eqn.Fij-psi<0}, we obtain
\[
Q:=\frac{1}{\g}\sum_{i=1}^nF_i(\vp_y)_i^2+2\b\sum_{i=1}^nF_i\l_i^2
\leq(k-\ell)+\a\sum_{i=1}^nF_i.
\]
Furthermore, set $\mathds 1:=(1,1,\dots,1)\in\R^n$.
For $F=\frac{\s_k}{\s_\ell}$ with $k-\ell\in\set{1,2}$,
since $F^{1/(k-\ell)}$ is concave and homogeneous of degree one in $\G_k$,
we have
\[
\frac{1}{c_*}:=F^{1/(k-\ell)}(\mathds1)
\leq F^{1/(k-\ell)}(\l)+DF^{1/(k-\ell)}(\l)\cdot(\mathds1-\l)
=\sum_{i=1}^nF^{1/(k-\ell)}_i(\l)
=\frac{1}{k-\ell}\sum_{i=1}^nF_i.
\]
Thus,
\begin{equation}\label{eqn.Q<Fi}
Q\leq\frb{\a+c_*}\sum_{i=1}^nF_i
=:\a_1\sum_{i=1}^nF_i.
\end{equation}

Since $|x-y|\geq r$, there exists $j\in\Z_{[1,n]}$ such that
\[
|x_j-y_j|\geq|x-y|/\sqrt n\geq r/\sqrt n.
\]
Set $z_j:=x_j-y_j-2\b u_j$.
Then $(\vp_y)_j=z_j\l_j+\a(x_j-y_j)$.
Choosing $\b:=\frac{r}{16M\sqrt n}$, we obtain
\[
2\b|u_j|
\leq2\b M
=\frac{r}{8\sqrt n}
\leq\frac18|x_j-y_j|,
\]
and therefore,
\begin{gather}
\frb{x_j-y_j}z_j
=(x_j-y_j)^2-2\b\frb{x_j-y_j}u_j
\geq\frac{7}{8}|x_j-y_j|^2>0,
\label{eqn.zj-bound}\\
\frac{7}{8}|x_j-y_j|
\leq|z_j|
\leq\frac{9}{8}|x_j-y_j|.
\nonumber
\end{gather}
Next, we distinguish two cases.

\textit{Case 1. $\l_j\geq0$}.
By \eqref{eqn.zj-bound},
the two terms $z_j\l_j$ and $\a(x_j-y_j)$ have the same sign.
By $\a\geq1$ and Lemma \ref{lem.Fj-lambdaj},
there exists $c_1=c_1(n,k)>0$ such that
\[
Q
\geq\frac{F_j}{\g}(\vp_y)_j^2
\geq\frac{F_j}{\g}\frb{\frac{7}{8}\l_j+\a}^2|x_j-y_j|^2
\geq\frac{49r^2}{64n\g}F_j(1+\l_j^2)
\geq\frac{49r^2c_1}{64n\g}\sum_{i=1}^nF_i.
\]
We may choose $\g<\frac{49r^2c_1}{64n\a_1}$
such that $Q>\a_1\sum_{i=1}^nF_i$,
a contradiction to \eqref{eqn.Q<Fi}.

\textit{Case 2. $\l_j<0$}.
Take $\g$ with $\g\leq\frac{r^2}{2\b n}=\frac{8Mr}{\sqrt n}$.
Then $2\b\g\leq(x_j-y_j)^2$.
By Lemma \ref{lem.Fj-lambdaj},
there exists $c_2=c_2(n,k)>0$ such that
\begin{align*}
Q
\geq\frac{1}{\g}F_j(\vp_y)_j^2+2\b F_j\l_j^2
\geq\frac{2\b\a^2(x_j-y_j)^2}{z_j^2+2\b\g}F_j
\geq\frac{1}{2}\b\a^2F_j
\geq\frac{1}{2}\b\a^2 c_2\sum_{i=1}^nF_i.
\end{align*}
Choose $\a\geq\max\set{1,c_*,\frac{8}{\b c_2}}$.
Then $Q>\a_1\sum_{i=1}^nF_i$, a contradiction to \eqref{eqn.Q<Fi}.
The claim is proved.

Now, using the lower bound \eqref{eqn.psi-lower} of $\psi_y$, we have
\begin{align*}
c_0\max_{\ol{B_{2r(y)}}}a
\leq\max_{\ol{B_{2r(y)}}}a\psi_y
\leq\max_{\ol{\Om_y}}a\psi_y
\leq a(x_0)\psi_y(x_0).
\end{align*}
If $x_0\in\ol{B_r(y)}$,
then by the upper bound \eqref{eqn.psi-upper} of $\psi_y$,
we have $a(x_0)\psi_y(x_0)\leq C_0\max_{\ol{B_r(y)}}a$.
If $\l_{\max}(x_0)\leq K$, then $a(x_0)\psi_y(x_0)\leq C_0K^\d$.
Thus,
\[
\max_{\ol{B_{2r}(y)}}a\leq\frac{C_0}{c_0}\max\set{\max_{\ol{B_r(y)}}a,K^\d}.
\]
Furthermore, since
\[
1=F(\l)\leq F(\l_1\mathds 1)=\l_{\max}^{k-\ell} F(\mathds1),
\]
we have $\l_{\max}\geq c_*>0$.
Thus
\[
\sup_{B_{2r}(y)}\l_{\max}
\leq\frb{\frac{C_0}{c_0}}^{1/\d}\max\set{1,\frac{K}{c_*}}\sup_{B_r(y)}\l_{\max},
\]
which proves the desired estimate.
\end{proof}

\section{Interior Hessian estimates}

In this section, we complete the proof of
Theorem \ref{thm.Hessian-estimate} by a compactness argument.
We first introduce the following Alexandrov-type result,
which ensures that 
locally uniform limits of smooth $k$-convex solutions of \eqref{eqn.sigma-quotient} 
admit second-order expansions almost everywhere.
Chen's gradient estimate \cite{C15} allows us to
follow the argument of Shankar--Yuan \cite[Section 4]{SY25},
as in Fung \cite[Proposition 3.1]{F26b},
without imposing the restriction $2k>n$.
When $2k>n$, this conclusion already follows from
the Alexandrov-type theorem of Chaudhuri--Trudinger \cite[Theorem 1.1]{CT05},
which applies to general $k$-convex functions.

\begin{lemma}[Alexandrov-type differentiability]
\label{lemma.alexandrov-quotient}
Let $1\leq\ell<k\leq n$ and let $\Om\subset\R^n$ be open.
Suppose that $u_j\in C^3(\Om)$ are $k$-convex solutions 
of $\frac{\s_k(D^2u_j)}{\s_\ell(D^2u_j)}=1$ in $\Om$,
and that $u_j$ converges locally uniformly to $u$ in $\Om$.
Then $u$ is twice differentiable almost everywhere in $\Om$;
that is, for almost every $x\in\Om$, 
there is a quadratic polynomial $P$ such that
\[
\sup_{y\in B_r(x)}|u(y)-P(y)|=o(r^2).
\]
\end{lemma}

\begin{proof}
By Chen's gradient estimate \cite[Theorem 1.1]{C15}, we have
\[
|Du_j(z)|\leq C(n,k,\ell)\frac{\op{osc}_{B_\rho(z)}u_j}{\rho}\quad
\text{whenever}\ B_\rho(z)\Subset\Om.
\]
Since $u_j$ converges locally uniformly to $u$,
this estimate gives locally uniform bounds for $Du_j$.
Hence $u$ is locally Lipschitz. 
Moreover, by Rademacher's theorem,
$u$ is differentiable almost everywhere in $\Om$.
Following the interpolation argument in \cite[Lemma 2.6]{TW99} 
leading to \cite[Corollary 3.4]{TW99},
we obtain the weighted Lipschitz estimate
\[
\sup_{\substack{x,y\in B_r(z)\\x\neq y}}
d_{x,y}^{n+1}\frac{|u_j(x)-u_j(y)|}{|x-y|}
\leq C(n,k,\ell)\int_{B_r(z)}|u_j|
\]
for every ball $B_r(z)\Subset\Om$,
where $d_{x,y}:=\min\{d_x,d_y\}$ and $d_x:=\op{dist}(x,\partial B_r(z))$.
Passing to the locally uniform limit gives
\begin{equation}\label{eqn.alexandrov-lip}
\sup_{\substack{x,y\in B_r(z)\\x\neq y}}
d_{x,y}^{n+1}\frac{|u(x)-u(y)|}{|x-y|}
\leq C(n,k,\ell)\int_{B_r(z)}|u|.
\end{equation}
The estimate \eqref{eqn.alexandrov-lip} also holds for $u-L$
with the same constant $C(n,k,\ell)$ for every affine function $L$,
since subtracting $L$ leaves equation \eqref{eqn.sigma-quotient} unchanged.
Since $k\geq2$, the limit $u$ is $2$-convex in the viscosity sense.
By the argument of \cite[Theorem 2.4]{CT05},
its distributional Hessian $[D^2u]=[\mu^{ij}]$
is a symmetric matrix of locally finite signed Radon measures:
\[
\int_\Om u\,\varphi_{ij}\,dx
=\int_\Om\varphi\,d\mu^{ij}
\quad\text{for all }\varphi\in C_0^\infty(\Om).
\]
By Lebesgue differentiation, the conditions in
\cite[(4.2)--(4.4)]{SY25} hold almost everywhere in $\Om$.
Thus \cite[Lemma 4.1]{SY25} gives, for almost every $x\in\Om$,
a quadratic polynomial $P$ such that
\[
\frac{1}{|B_r|}\int_{B_r(x)}|u(y)-P(y)|\,dy=o(r^2).
\]
Using \eqref{eqn.alexandrov-lip}
and arguing as in \cite[proof of (4.6)]{SY25},
we obtain the corresponding estimate for the remainder $u-P$.
Finally, \cite[Lemma 4.2]{SY25} yields
\[
\sup_{y\in B_{r/2}(x)}|u(y)-P(y)|=o(r^2).
\]
Replacing $r$ by $2r$ completes the proof.
\end{proof}

\medskip

We now combine Lemma \ref{lemma.alexandrov-quotient}
with Savin's small perturbation theorem \cite{Sav07}
and the doubling inequality (Lemma \ref{lem.doubling})
to complete the proof of Theorem \ref{thm.Hessian-estimate}.

\begin{proof}[Proof of Theorem \ref{thm.Hessian-estimate}]
Suppose to the contrary that
there exists a sequence of
$k$-convex solutions $u_j\in C^4(B_1)$ of equation \eqref{eqn.sigma-quotient} 
such that $\norm{u_j}_{C^1(B_1)}\leq A$ and $\abs{D^2u_j(0)}\to+\infty$.
By the Arzel\`a--Ascoli theorem
and the closedness of viscosity solutions (cf. \cite[Lemma 6.1]{SY25}),
after passing to a subsequence,
$u_j$ converges uniformly to
a continuous viscosity solution $u$ of $F(D^2u)=1$ in $B_1$.
Moreover, since each $u_j$ is $k$-convex,
the limit $u$ is $k$-convex in the viscosity sense.
By Lemma \ref{lemma.alexandrov-quotient},
$u$ is twice differentiable almost everywhere in $B_1$.
Fix such a point $y\in B_{1/8}$ and let $P$ be a quadratic polynomial
such that
\[
u(y+z)=P(y+z)+o(|z|^2).
\]
Note that $\l(D^2P)\in\G_k$ and $F(D^2P)=1$.

\medskip

We next apply Savin's small perturbation theorem \cite{Sav07} to $v_j:=u_j-P$.
For $0<r<1/8$, we rescale near $y$ by
\[
v_{j,r}(z)
:=\frac{1}{r^2}\frb{u_j(y+rz)-P(y+rz)}
\quad\text{for any}\ z\in B_1.
\]
Then
\begin{align*}
\norm{v_{j,r}}_{L^\infty(B_1)}
&\leq\frac{1}{r^2}\frb{\norm{u_j-u}_{L^\infty(B_r(y))}+\norm{u-P}_{L^\infty(B_r(y))}}
\\
&\leq\frac{1}{r^2}\norm{u_j-u}_{L^\infty(B_r(y))}+\ve(r),
\end{align*}
for some modulus $\ve(r):=o(r^2)/r^2$.
Define a continuous operator on the space of real symmetric matrices by
\[
G(M):=
\begin{cases}
F(D^2P+M)-1, & \l(D^2P+M)\in\G_k,\\
-1, & \text{otherwise}.
\end{cases}
\]
The operator $G$ is degenerate elliptic, namely,
\[
G(M+N)\geq G(M)
\quad\text{for any}\ M\in\op{Sym(n)},\ N\geq0.
\]
Moreover, $G(0)=0$, and $G$ is smooth and uniformly elliptic
in a neighborhood of $0$.
Since $\l(D^2P+D^2v_{j,r})\in\G_k$,
we have
\[
G(D^2v_{j,r})=F(D^2P+D^2v_{j,r})-1=0
\quad\text{in}\ B_1.
\]
Fix $r=r(n,k,P,\ve)=:\rho\in(0,1/8)$ sufficiently small that $\ve(\rho)<c_1/2$,
where $c_1$ is the small constant in \cite[Theorem 1.3]{Sav07}.
Since $u_j$ converges uniformly to $u$,
$\|v_{j,\rho}\|_{L^\infty(B_1)}\leq c_1$
for all sufficiently large $j$.
Applying \cite[Theorem 1.3]{Sav07} and returning to the original variables, we obtain
\[
\norm{u_j-P}_{C^{2,\a}(B_{\rho/2}(y))}
\leq C(n,k,P,\ve),
\]
for all sufficiently large $j$,
where $\a=\a(n,k,P,\ve)\in(0,1)$.
This implies that
\[
\l_{\max}(D^2(u_j))\leq C(n,k,P,\ve)=:C_0
\quad\text{in}\ B_{\rho/2}(y).
\]

\medskip

Finally, choose points $y_0:=y$, $y_1,y_2,\dots,y_{N-1}$, $y_N:=0$
on the line segment from $y$ to $0$
such that $\abs{y_i-y_{i-1}}\leq\rho/2$ for all $i\in\Z_{[1,N]}$
and $N\leq1+\frac{1}{4\rho}$.
Then
\[
\set{y_i}_{i=0}^N\subset B_{1/4},
\quad
B_{\rho/2}(y_i)\subset B_\rho(y_{i-1})
\ \text{for all}\ i\in\Z_{[1,N]}.
\]
Repeated application of the doubling inequality (Lemma \ref{lem.doubling}) to $u_j$  gives
\[
\begin{aligned}
\l_{\max}(D^2u_j(0))
&\leq
\sup_{B_{\rho/2}(y_N)}\l_{\max}(D^2u_j)
\\
&\leq
C\frb{n,k,A,\rho}^N
\sup_{B_{\rho/2}(y)}\l_{\max}(D^2u_j)
\\
&\leq C\frb{n,k,A,\rho}C_0
\end{aligned}
\]
for all sufficiently large $j$.
Since $\G_k\subseteq\G_2$, we have
\[
\abs{D^2u_j}^2
=(\D u_j)^2-2\s_2(D^2u_j)
<(\D u_j)^2\leq (n\l_{\max}(D^2u_j))^2.
\]
Consequently,
\[
\abs{D^2u_j(0)}
\leq nC(n,k,A,\rho)C_0
\]
for all sufficiently large $j$,
contradicting $\abs{D^2u_j(0)}\ra+\infty$.
This completes the proof.
\end{proof}

\subsection*{Acknowledgments}

I am sincerely grateful to Professor Zhisu Li
for his longstanding guidance, encouragement and support
and for many helpful discussions of this subject.
I also thank Cheuk Yan Fung for his helpful suggestion to use
Chen's gradient estimate in the argument of Shankar--Yuan \cite{SY25}
to obtain the Alexandrov-type result in Lemma~\ref{lemma.alexandrov-quotient}
without the restriction $2k>n$.
This work was partially supported by NSFC 12171389 and NSFC 11801015.

\renewcommand\refname{References}


\begin{thebibliography}{W26}


\bibitem[BCD17]{BCD17}
S. Brendle, K. Choi, and P. Daskalopoulos,
Asymptotic behavior of flows by powers of the Gaussian curvature,
Acta Math., 219 (2017), 1--16.

\bibitem[C15]{C15}
C.-Q. Chen,
The interior gradient estimate of Hessian quotient equations,
J. Differential Equations, 259 (2015), 1014--1023.

\bibitem[CHO16]{CHO16}
C.-Q. Chen, F. Han, and Q.-Z. Ou,
The interior $C^2$ estimate for the Monge-Amp\`ere equation in dimension $n=2$,
Anal. PDE, 9 (2016), 1419--1432.

\bibitem[CJTZ26]{CJTZ26}
R.-S. Chen, H.-Y. Jian, X.-S. Tu, and X.-C. Zhou,
Regularity for convex viscosity solutions of sigma-2 equation,
arXiv:2605.30823 (2026).

\bibitem[CNS85]{CNS85}
L. Caffarelli, L. Nirenberg, and J. Spruck,
The Dirichlet problem for nonlinear second-order elliptic equations.
III. Functions of the eigenvalues of the Hessian,
Acta Math., 155 (1985), 261--301.

\bibitem[CW01]{CW01}
K.-S. Chou and X.-J. Wang,
A variational theory of the Hessian equation,
Comm. Pure Appl. Math., 54 no. 9 (2001), 1029--1064.

\bibitem[CWY09]{CWY09}
J.-Y. Chen, M. Warren, and Y. Yuan,
A priori estimate for convex solutions to special Lagrangian equations and its application,
Comm. Pure Appl. Math., 62 (2009), 583--595.

\bibitem[CT05]{CT05}
N. Chaudhuri and N. S. Trudinger,
An Alexandrov type theorem for $k$-convex functions,
Bull. Austral. Math. Soc., 71 (2005), 305--314.

\bibitem[CZZ26]{CZZ26}
R.-S. Chen, X.-C. Zhou and R.-X. Zhu,
Interior $C^{2,\a}$ regularity for the quadratic Hessian equation,
arXiv:2608.29484 (2026).

\bibitem[DXZ26]{DXZ26}
W.-S. Dong, S.-R. Xu and R.-J. Zhang,
Pogorelov interior estimates for sum-of-Hessians equations,
arXiv:2603.15345v2 (2026).

\bibitem[DZ26]{DZ26}
W.-S. Dong and R.-J. Zhang,
Interior Hessian estimates for Hessian quotient equations,
arXiv:2608.19087v2 (2026).

\bibitem[Fan26]{Fan26}
Z.-Y. Fan,
Hessian estimates for the sigma-$2$ equation
with variable right-hand side terms in dimension $4$,
Adv. Math., 494 (2026), Paper No. 110953.

\bibitem[F26a]{F26a}
C. Y. Fung,
Doubling argument of the Hessian estimate for the Hessian quotient equations,
arXiv:2607.21982 (2026).

\bibitem[F26b]{F26b}
C. Y. Fung,
Doubling argument of the Hessian estimate 
for the special Lagrangian equation 
on general phases with constraints,
J. Differential Equations, 482 (2026), 114687.

\bibitem[GQ19]{GQ19}
P.-F. Guan and G.-H. Qiu,
Interior $C^2$ regularity of convex solutions
to prescribing scalar curvature equations,
Duke Math. J., 168 (2019), 1641--1663.

\bibitem[GS26]{GS26}
P.-F. Guan and M. Sroka,
A special concavity property for positive Hessian quotient operators,
Discrete Contin. Dyn. Syst., 54 (2026), 50--60.

\bibitem[K87]{K87}
N. J. Korevaar,
A priori interior gradient bounds for solutions to elliptic Weingarten equations,
Ann. Inst. H. Poincar\'e Anal. Non Lin\'eaire, 4 no. 5 (1987), 405--421.

\bibitem[H59]{H59}
E. Heinz,
On elliptic Monge--Amp\`ere equations and Weyl's embedding problem,
J. Analyse Math., 7 (1959), 1--52.


\bibitem[JS26]{JS26}
H.-M. Jiao and Z.-A. Sui,
Interior Hessian estimates for Hessian quotient equations in dimension three,
arXiv:2602.14064 (2026).

\bibitem[Liu21]{Liu21}
J.-K. Liu,
Interior $C^2$ estimate for Monge--Amp\`ere equation in dimension two,
Proc. Amer. Math. Soc., 149 (2021), 2479--2486.

\bibitem[Lu23]{Lu23}
S.-Y. Lu,
Interior $C^2$ estimate for Hessian quotient equation in dimension three,
arXiv:2311.05835 (2023).

\bibitem[Lu25]{Lu25}
S.-Y. Lu,
Interior $C^2$ estimate for Hessian quotient equation in general dimension,
Ann. PDE, 11 (2025), Paper No. 17, 26 pp.

\bibitem[LT26]{LT26}
S.-Y. Lu and Y.-L. Tsai,
A note on interior $C^2$ estimate for general Hessian quotient equation,
Commun. Pure Appl. Anal., 33 (2026), 88--100.

\bibitem[LW26a]{LW26a}
Z.-S. Li and K. Wu,
A concavity inequality and interior $C^2$ estimate for Hessian quotient equations,
arXiv:2608.17405 (2026).

\bibitem[LW26b]{LW26b}
Z.-S. Li and K. Wu,
Interior Hessian estimates for the quadratic Hessian equation,
arXiv:2608.23233v2 (2026).

\bibitem[MSY19]{MSY19}
M. McGonagle, C. Song, and Y. Yuan,
Hessian estimates for convex solutions to quadratic Hessian equation,
Ann. Inst. H. Poincar\'e C Anal. Non Lin\'eaire,
36 (2019), 451--454.

\bibitem[M21]{M21}
C. Mooney,
Strict $2$-convexity of convex solutions
to the quadratic Hessian equation,
Proc. Amer. Math. Soc., 149 (2021), 2473--2477.

\bibitem[MY26]{MY26}
X.-Q. Mei and J. Yan,
Interior $C^2$ estimate for semi-convex solutions to a class of
Hessian quotient equations in arbitrary dimensions,
arXiv:2604.23349 (2026).

\bibitem[P78]{P78}
A. V. Pogorelov,
The Minkowski multidimensional problem,
Halsted Press, New York--Toronto--London, 1978.

\bibitem[Q24]{Q24}
G.-H. Qiu,
Interior Hessian estimates for $\s_2$ equations in dimension three,
Front. Math., 19 (2024), 577--598.

\bibitem[QY26]{QY26}
G.-H. Qiu and J. Yan,
Interior estimates for Hessian quotient equations,
arXiv:2609.24549 (2026).

\bibitem[RW23]{RW23}
C.-Y. Ren and Z.-Z. Wang,
The global curvature estimate for the $n-2$ Hessian equation,
Calc. Var. Partial Differential Equations,
62 (2023), Paper No. 239.

\bibitem[Sav07]{Sav07}
O. Savin,
Small perturbation solutions for elliptic equations,
Comm. Partial Differential Equations, 32 (2007), 557--578.


\bibitem[S26]{S26}
R. Shankar,
Hessian estimates for the special Lagrangian equation by doubling,
Anal. PDE, 19 (2026), 339--352.

\bibitem[SY20]{SY20}
R. Shankar and Y. Yuan,
Hessian estimate for semiconvex solutions to the sigma-$2$ equation,
Calc. Var. Partial Differential Equations,
59 (2020), Paper No. 30, 12 pp.


\bibitem[SY25]{SY25}
R. Shankar and Y. Yuan,
Hessian estimates for the sigma-$2$ equation in dimension four,
Ann. of Math., 201 (2025), 489--513.

\bibitem[T26]{T26}
Y.-L. Tsai,
A concavity inequality for Hessian quotient equations,
arXiv:2608.16383 (2026).

\bibitem[TW99]{TW99} 
N. S. Trudinger and X.-J. Wang,
Hessian measures II,
Ann. of Math., 150 (1999), 579--604.

\bibitem[U90]{U90}
J. I. E. Urbas,
On the existence of nonclassical solutions
for two classes of fully nonlinear elliptic equations,
Indiana Univ. Math. J., 39 (1990), 355--382.

\bibitem[WY09]{WY09}
M. Warren and Y. Yuan,
Hessian estimates for the sigma-$2$ equation in dimension $3$,
Comm. Pure Appl. Math., 62 (2009), 305--321.

\bibitem[WY14]{WY14}
D.-K. Wang and Y. Yuan,
Hessian estimates for special Lagrangian equations
with critical and supercritical phases in general dimensions,
Amer. J. Math., 136 (2014), 481--499.

\bibitem[Yan26]{Yan26}
J. Yan,
Global curvature estimates for $\s_k$ curvature equations with $k\geq n/2$,
arXiv:2608.25665 (2026).


\bibitem[Z24]{Z24}
X.-C. Zhou,
Notes on generalized special Lagrangian equation,
Calc. Var. Partial Differential Equations, 63 (2024), Paper No. 197, 28 pp.

\bibitem[ZZ26]{ZZ26}
X.-C. Zhou and R.-X. Zhu,
Interior $C^{2,\a}$ regularity for convex solutions of the $2$-Hessian equation,
arXiv:2608.24604 (2026).
\end{thebibliography}
\end{document}